\documentclass[11pt,onecolumn]{article}
\usepackage{txfonts}
\usepackage{graphicx}
\usepackage{color}
\usepackage{ifpdf}

\usepackage{stmaryrd}
\usepackage{tipa}
\usepackage{tipa}
\usepackage{pifont}
\usepackage{bbding}
\usepackage{galois}

\usepackage{amssymb}
\usepackage{mathrsfs}
\usepackage{amsfonts}
\usepackage{amsthm,amscd,amsmath}
\usepackage{eufrak}
\usepackage{youngtab}
\usepackage{graphicx}
\usepackage[all]{xy}
\usepackage{indentfirst}

\allowdisplaybreaks

\makeatletter
\def\@biblabel#1{#1}
\makeatother

\newtheorem{theorem}{Theorem}[section]
\newtheorem{lemma}[theorem]{Lemma}

\theoremstyle{definition}

\theoremstyle{proposition}
\newtheorem{proposition}[theorem]{Proposition}

\theoremstyle{remark}

\theoremstyle{example}

\theoremstyle{corollary}

\date{}

\def\comp{\ensuremath\mathop{\scalebox{.6}{$\circ$}}}

\begin{document}

\title{Crystal Structure on the Category of Modules over Colored Planar Rook Algebra}

\author{Bin Li \\
\tiny{School of Mathematics and Statistics, Wuhan University, P.R. China} \\
\tiny{libin117@whu.edu.cn}}

\maketitle

\begin{abstract}
Colored planar rook algebra is a semigroup algebra in which the basis element has a diagrammatic description.
The category of finite dimensional modules over this algebra is completely reducible and suitable functors
are defined on this category so that it admits a crystal structure in the sense of Kashiwara. We show that
the category and functors categorify the crystal bases for the polynomial representations of quantized enveloping algebra $U_q(gl_{n+1})$.
\end{abstract}

\section{Introduction}

Let $gl_{n+1}(\mathbb C)$ be the general Linear Lie algebra and let $\mathfrak h$ be the Cartan subalgebra of $gl_{n+1}(\mathbb C)$
 consisting of all diagonal $(n+1)\times(n+1)$ matrices over $\mathbb C$. Take $\epsilon_i\in\mathfrak h^{\ast}$ such that $\epsilon_i(E_{jj})=\delta_{ij}$ for 
 $1\leqslant i\leqslant n+1$. 
 We denote by $P$ and $Q$ the weight lattice and the root lattice respectively, i.e.
 \begin{align*}
 P&=\{\mu=\sum_{i=1}^{n+1}\mu_i\epsilon_i\ |\ \mu_i\in\mathbb Z\},\\
 Q&=\{\sum_{i=1}^{n}\mu_i(\epsilon_i-\epsilon_{i+1})\ |\ \mu_i\in\mathbb Z\}.
 \end{align*}
A set $B$ with maps $(wt, \tilde{e}_i, \tilde{f}_i,
\varepsilon_i, \varphi_i)$ 
\begin{align*}
wt&:\ B\longrightarrow P,\\
\varepsilon_i,\ \varphi_i&:\ B \longrightarrow \Bbb{Z} \cup \{-\infty \},\\
\tilde{e}_i,\ \tilde{f}_i&:\ B\longrightarrow B\cup \{0\} 
\end{align*}
for $1\leqslant i\leqslant n$, is called a $gl_{n+1}(\mathbb C)$-crystal if the following axioms are satisfied:
\begin{itemize}
\item[(a)]
For $b\in B,$ $\varphi_i(b)=\varepsilon_i(b)+
wt(b)(E_{ii}-E_{i+1,i+1}).$
\item[(b)]
For $b\in B$ with
$\tilde{e}_ib\in B$, $wt(\tilde{e}_ib)=wt(b)+\epsilon_i-\epsilon_{i+1}$;\\
For $b\in B$ with $\tilde{f}_ib\in B,\
wt(\tilde{f}_ib)=wt(b)-\epsilon_i+\epsilon_{i+1}.$
\item[(c)]
For $b_1, b_2\in B$,
$\tilde{f}_ib_2=b_1$ if and only if $\tilde{e}_ib_1=b_2$.
\item[(d)]
If $\varepsilon_i(b)=-\infty$,
then $\tilde{e}_ib=\tilde{f}_ib=0.$
\end{itemize}
The crystal graph associated to $B$ is a graph with the set of vertices $B$ and there is an arrow colored by $i$ from $b_1\in B$ to $b_2\in B$
if $\tilde{f}_i(b_1)=b_2$. In a series of works\cite{Kashiwara:1990, Kashiwara:1991, Kashiwara:1993, Kashiwara2:1993, Kashiwara:1994}, Kashiwara developed the crystal base theory to study the representation and the structure of
the quantized enveloping algebra $U_q(\mathfrak g)$ associated to a Kac-Moody algebra $\mathfrak g$. For any irreducible highest weight
integrable $U_q(\mathfrak g)$-module $V$, Kashiwara found a basis for $V$ which admits a $\mathfrak g$-crystal structure when $q=0$. 
Take $\mathfrak g=gl_{n+1}(\mathbb C)$ as an example. If $V_1$ and $V_2$ are polynomial $U_q(gl_{n+1})$-modules
with crystal bases $B_1$ and $B_2$, it is known that
$V_1\otimes V_2$ has a crystal basis $B_1\otimes B_2=\{b_1\otimes b_2\ |\ b_i\in B_i\}$ with
\begin{align*}
\tilde{e_i}(b_1\otimes b_2)&=
\begin{cases}
\tilde{e}_i(b_1)\otimes b_2& \text{if $\varphi_i(b_1)\geqslant\varepsilon_i(b_2)$},\\
b_1\otimes\tilde{e}_i(b_2)& \text{if $\varphi_i(b_1)<\varepsilon_i(b_2)$},
\end{cases}\\
\tilde{f_i}(b_1\otimes b_2)&=
\begin{cases}
\tilde{f}_i(b_1)\otimes b_2& \text{if $\varphi_i(b_1)>\varepsilon_i(b_2)$},\\
b_1\otimes\tilde{f}_i(b_2)& \text{if $\varphi_i(b_1)\leqslant\varepsilon_i(b_2)$},
\end{cases}\\
wt(b_1\otimes b_2)&=wt(b_1)+wt(b_2),\\
\varepsilon_i(b_1\otimes b_2)&=\textrm{max}\{\varepsilon_i(b_1),\varepsilon_i(b_2)-wt(b_1)(E_{ii}-E_{i+1,i+1})\},\\
\varphi_i(b_1\otimes b_2)&=\textrm{max}\{\varphi_i(b_2),\varphi_i(b_1)+wt(b_2)(E_{ii}-E_{i+1,i+1})\}.
\end{align*}
Moreover, the decomposition of $V_1\otimes V_2$ into a direct sum of irreducible submodules is compatible with the decomposition
of $B_1\otimes B_2$ into connected components. Hence crystal basis theory works as an important combinatorial tool to study the 
representation theory of the quantized enveloping algebra. 

Various realizations of the crystal bases are also considered by mathematicians such as Young tableaux realization, Littelmann's path model or geometric
models \cite{HK, Kashiwara:2003, Littelmann1, Littelmann2, Savage}. In this 
paper, by studying the colored planar rook algebra and its representations, we realize the highest weight crystal basis $B(\lambda)$ for irreducible
 polynomial $U_q(\mathfrak g)$-module $V(\lambda)$. The colored planar rook algebra $\mathbb CP_m$ is first defined and studied by Mousley, Schley and Shoemaker
 in \cite{MSS} to give an alternative proof of the recursive formula for multinomial coefficients
 \begin{equation*}
 \binom{m}{m_0,m_1,\cdots,m_n}=\sum_{i=0}^{n}\binom{m-1}{m_0,\cdots,m_i-1,\cdots,m_n}.
 \end{equation*}
They describe  all irreducible finite dimensional modules over the colored planar rook algebra $\mathbb CP_m$ in \cite{MSS} which generalizes the work
of Flath, Halverson and
Herbig\cite{FHK}. The above equality is 
 proved then by taking the dimension of certain $\mathbb CP_m$-module which restricts to a $\mathbb CP_{m-1}$-module through an embedding
 \begin{equation*}
 \mathbb CP_{m-1}\longrightarrow \mathbb CP_m. 
 \end{equation*} 
 Their work inspires the author to consider other sorts of embeddings from $\mathbb CP_{m-1}$ to $\mathbb CP_{m}$ which help to define suitable functors
 on the category of $\mathbb CP_m$-modules and finally realize the crystal. 
 
We organize the paper as follows. In section 2, we recall the definition of the the colored planar rook algebra and summarize the main 
results in \cite{MSS}. In section 3, restriction and induction functors are defined through various embeddings $\mathbb CP_{m-1}\longrightarrow
\mathbb CP_{m-1}$ and their properties are investigated. In the last section, we show that the category of finite dimensional $\mathbb CP_m$-modules
categorifies the crystal $B(m\epsilon_1)$, and furthermore, by studying the representations of the tensor product 
$\mathbb CP_{\lambda_1}\otimes\cdots\otimes\mathbb CP_{\lambda_k}$, we realize the crystal $B(\lambda)$.

\section{Colored Planar Rook Algebra}

We fix a positive integer $n$ and denote by $\mathbb Z_2^n$ the direct product of $n$ copies of the ring $\mathbb Z_2$. Take $u_i\in\mathbb Z_2^n$ whose $i^{th}$ component is 1 and the others are 0. Given a positive integer $m$, let $P_m^n$, or $P_m$ for simplicity, denote the set consisting of all $m\times m$ matrices over $\mathbb Z_2^n$ which satisfies the following conditions,
\begin{itemize}
\item[(a)] there is at most one none-zero element from the set $\{u_i\ |\ 1\leqslant i\leqslant n\}$ in each row and each column.
\item[(b)] there is no $2\times 2$ submatrix of the following form for any $1\leqslant i\leqslant n$,
\begin{displaymath}
\left(\begin{array}{cc}
0 & u_i  \\
u_i& 0\\
\end{array} \right).
\end{displaymath}
\end{itemize}
For example, the matrix $B$ below is an element in $P_5$,
\begin{displaymath}
B=\left(\begin{array}{ccccc}
0 & u_1 & 0  & 0  & 0 \\
u_2 & 0 & 0  & 0  & 0 \\
0 & 0 & u_1 & 0  & 0 \\
0 & 0 & 0 & 0 & u_2 \\
0 & 0 & 0 &  0 &  0 \\
\end{array} \right).
\end{displaymath}
It can be easily checked that $P_m$ is a semigroup under matrix multiplication. To each element $A=(a_{ij})_{m\times m}$ in $P_m$, one can associate a diagram  consisting of two rows of $m$ vertices, where the $i^{th}$ vertex in the top row and the $j^{th}$ vertex in the bottom row are connected by an edge colored with $1\leqslant k\leqslant n$  if $a_{ij}=u_k$. Thus the diagram corresponding to $B$ above is as follows, where  red and blue are used for the color $1$ and $2$ respectively.
\begin{figure}[h]
\centering
\includegraphics[width=36mm]{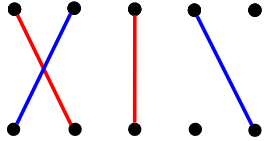}
\caption{the planar diagram associated to $B$}
\end{figure}

It is required by (b) that the edges with the same color cannot cross in such so-called \textit{planar} diagrams. We also mention that given 2 planar diagrams $d_1$ and $d_2$, the product $d_1d_2$ is the diagram  obtained by stacking $d_1$ on top of $d_2$, and then removing the concatenation of 2 edges with different colors or the edges connecting to isolated vertices. See that $d_1d_2$ is still planar and the associated matrix is exactly the product $D_1D_2$ of two matrices where $D_i\in P_m$ corresponds to the diagram $d_i$ for $i=1, 2$. Thus we do not distinguish the matrices in $P_m$ and the planar diagrams hereafter.

The \textit{n colored planar rook algebra} $\mathbb CP_m^n$, or $\mathbb CP_m$, is then defined to be the semigroup algebra  associated to $P_m$, i.e. it consists of all formal expressions
 \begin{equation*}
 \sum_{d\in P_m}a_dd
 \end{equation*}
 where $a_d\in\mathbb C$
 and the multiplication is given by linearly extending the product in $P_m$. Note that there is an embedding of algebras from $\mathbb CP_{m_1}\otimes \mathbb CP_{m_2}$ to $\mathbb CP_{m_1+m_2}$ through which, for $d_i\in P_{m_i}$, we regard  $d_1\otimes d_2$ as the planar diagram in $P_{m_1+m_2}$ by putting $d_1$ and $d_2$ next to each other. For $1\leqslant i\leqslant n$, let $I_i$ denote the diagram in $P_1$ consisting of two vertices connected by an edge colored $i$ and let $I_0\in P_1$ be the single pair of isolated vertices. Then the identity element in $\mathbb CP_m$ is
 \begin{equation*}
 e_m=\underbrace{e_1\otimes\cdots \otimes e_1}_{m},
 \end{equation*}
where $e_1=\sum_{i=1}^{n}I_i-(n-1)I_0\in\mathbb CP_1$.

In this paper, we consider only (left) $\mathbb CP_m$-modules, or called representations
of $\mathbb CP_m$ which are unital, i.e. given a $\mathbb CP_m$-module $M$, $e_mx=x$ for
all $x\in M$. To decompose the regular representation of $\mathbb CP_m$, a new basis for
the algebra is given as the following
\begin{equation*}
\{x_d=\sum_{d^{\prime}\subseteq d}(-1)^{size(d)-size(d^{\prime})}d^{\prime}\in\mathbb CP_m\  |\ d\in P_m \},
\end{equation*}
where $d^{\prime}\subseteq d$ means $d^{\prime}$ is a sub-diagram of $d$ obtained by deleting
 some of the edges in $d$ and $size(d)$ denotes the number of edges in $d$.

For $d\in P_m$ and $1\leqslant i\leqslant n$, let $\tau_i(d)$ (resp. $\beta_i(d)$) be the set
 of all vertices in the top (resp. bottom) row of $d$ which is connected by an edge
 colored $i$. Also we denote by $\tau_0(d)$ (resp. $\beta_0(d)$) the set of all
 isolated top (resp. bottom) vertices in $d$. Set $\tau(d)=(\tau_0(d), \tau_1(d), \cdots, \tau_n(d))$
  and $\beta(d)=(\beta_0(d), \beta_1(d), \cdots, \beta_n(d))$. Taking $B\in P_5^2$ in
  Figure 1 for example, one has
\begin{equation*}
\tau(B)=(\{5\},\{1,3\},\{2,4\}),\quad \beta(B)=(\{4\},\{2,3\},\{1,5\}).
\end{equation*}
Let $X_m=\{\tau(d)\ | \ d\in P_m \}$. For $T=(T_{0}, \cdots, T_{n}), S=(S_{0}, \cdots, S_{n})\in X_m$,
 we say that $T\subseteq S$ if $T_{i}\subseteq S_{i}$ for all $1\leqslant i\leqslant n$.
\begin{proposition}(\cite{MSS})~\label{product}
For $d, d^{\prime}\in P_m$,
\begin{align*}
d^{\prime}\cdot x_d&=
\begin{cases}
x_{d^{\prime}d}& \text{if $\tau(d)\subseteq\beta(d^{\prime})$},\\
0 & \text{otherwise},
\end{cases}\\
x_{d^{\prime}}\cdot d&=
\begin{cases}
x_{d^{\prime}d}& \text{if $\beta(d^{\prime})\subseteq\tau(d)$},\\
0 & \text{otherwise},
\end{cases}\\
x_{d^{\prime}}\cdot x_d&=
\begin{cases}
x_{d^{\prime}d}& \text{if $\beta(d^{\prime})=\tau(d)$},\\
0 & \text{otherwise}.
\end{cases}
\end{align*}
\end{proposition}
For $T\in X_m$, let $W^m_T$ denote the subspace of $\mathbb CP_m$ spanned
 by all $x_d$ with $\beta(d)=T$. The following theorem follows easily from
Proposition~\ref{product}.
\begin{theorem}(\cite{MSS})~\label{semisimple}
\begin{itemize}
\item[$(1)$]For $T\in X_m$, $W^m_T$ is an irreducible $\mathbb CP_m$-module.
 Moreover, $\mathbb CP_m$ is semisimple, i.e. the regular module
  $\mathbb CP_m=\oplus_{T\in X_m}W^m_T$.
\item[$(2)$] For $T, T^{\prime}\in X_m$, $W^m_T\cong W^m_{T^{\prime}}$
 as $\mathbb CP_m$-modules if and only if $|T_i|=|T_i^{\prime}|$ for all $0\leqslant i\leqslant n$.
\end{itemize}
\end{theorem}
For $d\in P_m$, we define $\omega(d)$ to be the planar diagram in $P_m$ by flipping $d$ about
 the horizonal axis, i.e. $\omega(d)$ is the transpose of $d$ when viewed as a  matrix.
 We denote by $\omega$ the linear transformation of $\mathbb CP_m$ by linearly extending the map $d\longmapsto\omega(d)$.
 One can check that $\omega$ is an anti-involution and
 \begin{equation}~\label{omega}
 \omega(d_1\otimes d_2)=\omega(d_1)\otimes\omega(d_2),\quad
 \omega(x_d)=x_{\omega(d)}.
 \end{equation}

\section{Restriction and Induction Functors}
Let $\mathscr{C}_m$ denote the category of finite dimensional $\mathbb CP_m$-modules. By the well-known Artin-Wedderburn theorem and  Theorem~\ref{semisimple}, $\mathscr{C}_m$ is completely reducible and each irreducible $\mathbb CP_m$-module is isomorphic to some $W^m_T$. For $T\in X_m$, let
\begin{equation*}
\overline{T}=(\{1,2,\cdots,m_0\},\{m_0+1,\cdots,m_0+m_1\},\cdots,\{\sum_{i=0}^{n-1}m_i,\cdots,m\})\in X_m
\end{equation*} where $m_i=|T_i|$ for $0\leqslant i\leqslant n$. Then $W^m_{\overline{T}}$, or denoted by $W^m_{m_0,\cdots,m_n}$, can be naturally chosen as a representative of the isomorphism class $[W^m_T]$.

Apart from the embedding given in \cite{MSS}, we consider, for each $1\leqslant i\leqslant n$, the embedding of algebras
\begin{equation*}
\psi_{i,m}: \mathbb CP_{m-1}\longrightarrow \mathbb CP_{m}
\end{equation*}
which takes $x\in \mathbb CP_{m-1}$ to $x\otimes (I_i-I_0)\in\mathbb CP_{m}$. Since $I_i-I_0$ is an idempotent in $\mathbb CP_1$, the
map $\psi_{i,m}$ indeed preserves the product and we note also that $\psi_{i,m}$ is not unital,
i.e. the image of the identity element $e_{m-1}$ is not $e_{m}$ but an idempotent $e_{m-1}\otimes(I_i-I_0)$ instead.

Given a $\mathbb CP_m$-module $M$ in $\mathscr{C}_m$, it is equivalent to say that there is a homomorphism of algebras
\begin{equation*}
\phi: \mathbb CP_m\longrightarrow \textrm{End}_{\mathbb C}(M).
\end{equation*}
The composition of $\phi$ and $\psi_{i,m}$ is then again a homomorphism of algebras
\begin{equation*}
\phi\comp\psi_{i,m}: \mathbb CP_{m-1}\longrightarrow \textrm{End}_{\mathbb C}(M).
\end{equation*}
Mention that we cannot regard $M$ as a $\mathbb CP_{m-1}$-module through this map, since it is not unital. To obtain a $\mathbb CP_{m-1}$-module,
it needs a little modification. We denote the idempotent $e_{m-1}\otimes(I_i-I_0)$ by $e$ for simplicity.
Then $\mathbb CP_{m-1}$ is actually embedded into $e\mathbb CP_{m}e$ by $\psi_{i,m}$, i.e.
\begin{equation*}
\overline{\psi}_{i,m}: \mathbb CP_{m-1}\longrightarrow e\mathbb CP_{m}e,\quad x\longmapsto x\otimes (I_i-I_0).
\end{equation*}
Here $e$ becomes the identity element of the subalgebra $e\mathbb CP_{m}e$ and hence $\overline{\psi}_{i,m}$ is a unital homomorphism of algebras.
See that $eM$ is an
$e\mathbb CP_{m}e$-module, i.e. there is a homomorphism of algebras
\begin{equation*}
\overline{\phi}: e\mathbb CP_me\longrightarrow \textrm{End}_{\mathbb C}(eM).
\end{equation*}
Composing $\overline{\psi}_{i,m}$ and $\overline{\phi}$, one has
\begin{equation*}
\overline{\phi}\comp\overline{\psi}_{i,m}: \mathbb CP_{m-1}\longrightarrow \textrm{End}_{\mathbb C}(eM),
\end{equation*}
through which $eM$ can be seen as a $\mathbb CP_{m-1}$-module.
The restriction functor $Res_{i,m}$ is then defined as follows,
\begin{equation*}
Res_{i,m}:\mathscr{C}_m\longrightarrow \mathscr{C}_{m-1}
\end{equation*}
which takes each $M\in ob(\mathscr{C}_m)$ to $Res_{i,m}(M)\triangleq eM=e_{m-1}\otimes(I_i-I_0)M$.
Restricting each homomorphism $f:M\longrightarrow N$ of $\mathbb CP_m$-modules to $eM$, one obtains a homomorphism of $\mathbb CP_{m-1}$-modules
\begin{equation*}
Res_{i,m}(f):eM\longrightarrow eN.
\end{equation*}
\begin{lemma}~\label{res}
For $d\in P_m$ and $1\leqslant i\leqslant n$,
\begin{align}~\label{left}
(e_{m-1}\otimes(I_i-I_0))\cdot x_d=
\begin{cases}
x_d& \text{if $m\in\tau_i(d)$},\\
0 & \text{otherwise}.
\end{cases}
\end{align}
\end{lemma}
\begin{proof}
Suppose $m\in\tau_i(d)$. See that $e_{m-1}\otimes I_0$ is a linear combination of planar diagrams of the form
\begin{equation*}
I_{i_1}\otimes\cdots\otimes I_{i_{m-1}}\otimes I_0,
\end{equation*}
where $0\leqslant i_j\leqslant n$ for $1\leqslant j\leqslant m-1$. Let $d^{\prime}=I_{i_1}\otimes\cdots\otimes I_{i_{m-1}}\otimes I_0\in P_m$
Clearly $\tau(d^{\prime})=\beta(d^{\prime})$ and
$m\in\beta_0(d^{\prime})$. Hence $m\notin\beta_i(d^{\prime})\nsupseteq\tau_i(d)$, which implies by Proposition~\ref{product} that
\begin{equation*}
d^{\prime}x_d=0.
\end{equation*}
Then we have $(e_{m-1}\otimes I_0)\cdot x_d=0$. It can be similarly proved that $(e_{m-1}\otimes I_j)\cdot x_d=0$
for $1\leqslant j\leqslant n$ with $j\neq i$. Since $e_m=e_{m-1}\otimes(\sum_{k=1}^{n}I_k-(n-1)I_0)$, then we have
\begin{equation*}
e_m\cdot x_d=(e_{m-1}\otimes I_i)\cdot x_d=x_d.
\end{equation*}
Hence $(e_{m-1}\otimes(I_i-I_0))\cdot x_d=x_d$.

If $m\notin\tau_i(d)$, then $m\in\tau_j(d)$ for some $0\leqslant j\leqslant n$ with $j\neq i$. First assume $j\neq 0$.
One can prove similarly as above that
\begin{equation*}
(e_{m-1}\otimes I_i)\cdot x_d=(e_{m-1}\otimes I_0)\cdot x_d=0,
\end{equation*}
and the conclusion follows. Second, assume that $j=0$. We write $e_{m-1}\otimes(I_i-I_0)$ as a linear combination of vectors of the form
\begin{equation*}
I_{i_1}\otimes\cdots\otimes I_{i_{m-1}}\otimes(I_i-I_0).
\end{equation*}
It is easy to check that $\beta(I_{i_1}\otimes\cdots\otimes I_{i_{m-1}}\otimes I_i)\supseteq\tau(d)$ if and only if
$\beta(I_{i_1}\otimes\cdots\otimes I_{i_{m-1}}\otimes I_0)\supseteq\tau(d)$. By Proposition~\ref{product} and the fact
$(I_{i_1}\otimes\cdots\otimes I_{i_{m-1}}\otimes I_i)\cdot d=(I_{i_1}\otimes\cdots\otimes I_{i_{m-1}}\otimes I_0)\cdot d$,
one has
\begin{equation*}
(I_{i_1}\otimes\cdots\otimes I_{i_{m-1}}\otimes I_i)\cdot x_d=(I_{i_1}\otimes\cdots\otimes I_{i_{m-1}}\otimes I_0)\cdot x_d,
\end{equation*}
which completes the proof.
\end{proof}
It follows from Lemma~\ref{res} that for $T\in X_m$ and $1\leqslant i\leqslant n$,
\begin{align*}
Res_{i,m}(W^m_{T})&=\mathbb C\textendash\textrm{span}\{ex_d\ |\ d\in P_m, \beta(d)=T\}\\
&=\mathbb C\textendash\textrm{span}\{x_d\ |\ d\in P_m, \beta(d)=T, m\in\tau_i(d)\}.
\end{align*}
\begin{theorem}~\label{reduction}
For $T\in X_m$ and $1\leqslant i\leqslant n$, $Res_{i,m}(W^m_{T})=0$ if $T_i=\emptyset$,
otherwise, there is an isomorphism of $\mathbb CP_{m-1}$-modules
\begin{equation*}
Res_{i,m}(W^m_{T})\cong W^{m-1}_{m_0,m_1,\cdots,m_i-1,m_i,\cdots,m_n},
\end{equation*}
where $m_i=|T_i|$.
\end{theorem}
\begin{proof}
It is obvious for the case $T_i=\emptyset$. Hence we assume that $T_i\neq\emptyset$.
Note that $T$ corresponds to a row of $m$ vertices colored by the set $\{0,1,\cdots, n\}$. Let $\hat{T}\in X_{m-1}$ correspond to the row
of $m-1$ vertices obtained by deleting the rightmost vertex colored by $i$. For $d\in P_m$
with $\beta(d)=T$ and $m\in\tau_i(d)$, we denote by $\hat{d}$ the planar diagram in $P_{m-1}$ by deleting the rightmost edge colored by
$i$ and the two vertices connected. Observe that $\beta(\hat{d})=\hat{T}$ and
\begin{equation*}
\tau(\hat{d})=(\tau_0(d),\cdots, \tau_i(d)\setminus\{m\},\cdots,\tau_n(d)).
\end{equation*}
Linearly extending the bijection $\gamma: x_d\longmapsto x_{\hat{d}}$, we have an isomorphism of $\mathbb C$-vector spaces
\begin{equation*}
\gamma: Res_{i,m}(W^m_{T})\longrightarrow W^{m-1}_{\hat{T}}.
\end{equation*}
Since $W^{m-1}_{\hat{T}}\cong W^{m-1}_{m_0,m_1,\cdots,m_i-1,m_i,\cdots,m_n}$ by Theorem~\ref{semisimple},
 it remains to prove that $\gamma$ commutes with the action of $\mathbb CP_{m-1}$. For $d^{\prime}\in P_{m-1}$
and $d\in P_m$ with $\beta(d)=T$, $m\in\tau_i(d)$,
\begin{align*}
\gamma(d^{\prime}x_d)&=\gamma(\psi_{i,m}(d^{\prime})\cdot x_d)\\
                    &=\gamma((d^{\prime}\otimes(I_i-I_0))\cdot x_d)\\
                    &=\gamma((d^{\prime}\otimes I_i)\cdot x_d),
\end{align*}
where the third equality follows from the fact $(d^{\prime}\otimes I_0)\cdot x_d=0$. See that
$\beta(d^{\prime}\otimes I_i)\supseteq\tau(d)$ if and only if $\beta(d^{\prime})\supseteq\tau(\hat{d})$. Hence
if $\beta(d^{\prime}\otimes I_i)\nsupseteq\tau(d)$,
\begin{equation*}
\gamma(d^{\prime}x_d)=\gamma(0)=0=d^{\prime}\cdot x_{\hat{d}}=d^{\prime}\cdot \gamma(x_d).
\end{equation*}
If $\beta(d^{\prime}\otimes I_i)\supseteq\tau(d)$, one has
\begin{equation*}
\gamma(d^{\prime}x_d)=\gamma(x_{(d^{\prime}\otimes I_i)\cdot d})=x_{d^{\prime}\cdot\hat{d}}=d^{\prime}x_{\hat{d}}=d^{\prime}\gamma(x_d).
\end{equation*}
\end{proof}
We remark here that a similar result was shown by Mousley, Schley and Shoemaker in \cite[Theorem 4.1]{MSS}
while there is a little difference which results from the different
definitions of embeddings $\mathbb CP_{m-1}\longrightarrow\mathbb CP_m$.

Next, for any embedding $\psi_{i,m}: \mathbb CP_{m-1}\longrightarrow\mathbb CP_m$, we define an induction functor
\begin{equation*}
Ind_{i,m}: \mathscr{C}_{m-1}\longrightarrow\mathscr{C}_{m}
\end{equation*}
by $Ind_{i,m}(M)=\mathbb CP_{m}\otimes_{\mathbb CP_{m-1}}M$ for $M\in ob(\mathscr{C}_{m-1})$. Here
\begin{align*}
Ind_{i,m}(M)&=\mathbb CP_{m}(e+(e_m-e))\otimes_{\mathbb CP_{m-1}}M\\
                     &=\mathbb CP_{m}e\otimes_{\mathbb CP_{m-1}}M+\mathbb CP_{m}(e-e_m)\otimes_{\mathbb CP_{m-1}}M\\
                     &=\mathbb CP_{m}e\otimes_{\mathbb CP_{m-1}}M+\mathbb CP_{m}(e-e_m)\otimes_{\mathbb CP_{m-1}}e_{m-1}M\\
                     &=\mathbb CP_{m}e\otimes_{\mathbb CP_{m-1}}M+\mathbb CP_{m}(e-e_m)\psi_{i,m}(e_{m-1})\otimes_{\mathbb CP_{m-1}}M\\
                     &=\mathbb CP_{m}e\otimes_{\mathbb CP_{m-1}}M+\mathbb CP_{m}(e-e_m)e\otimes_{\mathbb CP_{m-1}}M\\
                     &=\mathbb CP_{m}e\otimes_{\mathbb CP_{m-1}}M,
\end{align*}
where $e=e_{m-1}\otimes(I_i-I_0)$.
\begin{lemma}~\label{induce}
For $d\in P_m$ and $1\leqslant i\leqslant n$,
\begin{align}
x_d\cdot(e_{m-1}\otimes(I_i-I_0))=
\begin{cases}
x_d& \text{if $m\in\beta_i(d)$},\\
0 & \text{otherwise}.
\end{cases}
\end{align}~\label{right}
\end{lemma}
\begin{proof}
The result is straightforward by applying $\omega$ to (\ref{left}).
\end{proof}
Suppose $M=W^{m-1}_T$ for some $T\in X_{m-1}$. Let $d_T$ be the unique planar diagram with $\tau(d_T)=\beta(d_T)=T$.
We have $x_{d_T}\in W^{m-1}_T$ and
\begin{equation*}W^{m-1}_T=\mathbb CP_{m-1}x_{d_T}
\end{equation*} since $W^{m-1}_T$ is irreducible. Hence
\begin{align*}
Ind_{i,m}(W^{m-1}_T)&=\mathbb CP_{m}e\otimes_{\mathbb CP_{m-1}}W^{m-1}_T\\
                             &=\mathbb CP_{m}e\otimes_{\mathbb CP_{m-1}}\mathbb CP_{m-1}x_{d_T}\\
                             &=\mathbb CP_{m}e\otimes_{\mathbb CP_{m-1}}x_{d_T}\\
                             &=\mathbb C\textendash\textrm{span}\{x_d\otimes_{\mathbb CP_{m-1}}x_{d_T}\ |\ d\in P_m,\ m\in\beta_i(d)\},
\end{align*}
where the fourth equality is implied by Lemma~\ref{induce}.
\begin{lemma}~\label{tensor}
For $T\in P_m$ and $1\leqslant i\leqslant n$,
\begin{equation*}
x_{d_T}\otimes(I_i-I_0)=x_{d_T\otimes I_i}.
\end{equation*}
\end{lemma}
\begin{proof}
For $d^{\prime}\subseteq d_T\otimes I_i$, $d^{\prime}=d_1\otimes I_i$ or $d^{\prime}=d_2\otimes I_0$ where $d_1, d_2\subseteq d_T$. Thus
\begin{align*}
x_{d_T\otimes I_i}&=\sum_{d^{\prime}\subseteq d_T\otimes I_i}(-1)^{size(d_T\otimes I_i,d^{\prime})}d^{\prime}\\
                  &=\sum_{d_1\subseteq d_T}(-1)^{size(d_T,d_1)}d_1\otimes I_i+\sum_{d_2\subseteq d_T}(-1)^{size(d_T,d_2)+1}d_2\otimes I_0\\
                  &=\sum_{d_1\subseteq d_T}(-1)^{size(d_T,d_1)}d_1\otimes I_i-\sum_{d_1\subseteq d_T}(-1)^{size(d_T,d_1)}d_1\otimes I_0\\
                  &=\sum_{d_1\subseteq d_T}(-1)^{size(d_T,d_1)}d_1\otimes(I_i-I_0)\\
                  &=x_{d_T}\otimes(I_i-I_0),
\end{align*}
where we write $size(d_1,d_2)$ for $size(d_1)-size(d_2)$.
\end{proof}
By Proposition~\ref{product}, $x_{d_T}\cdot x_{d_T}=x_{d_T}$ for $T\in X_{m-1}$. Then we have, for $d\in P_m$ with $m\in\beta_i(d)$,
\begin{align*}
x_d\otimes_{\mathbb CP_{m-1}}x_{d_T}&=x_d\otimes_{\mathbb CP_{m-1}}x_{d_T}x_{d_T}\\
                                    &=x_d\psi_{i,m}(x_{d_T})\otimes_{\mathbb CP_{m-1}}x_{d_T}\\
                                    &=x_dx_{d_T\otimes I_i}\otimes_{\mathbb CP_{m-1}}x_{d_T}\\
                                    &=\delta_{\beta(d),\tau(d_T\otimes I_i)}x_{d\cdot(d_T\otimes I_i)}\otimes_{\mathbb CP_{m-1}}x_{d_T}\\
                                    &=\delta_{\beta(d),\tau(d_T\otimes I_i)}x_d\otimes_{\mathbb CP_{m-1}}x_{d_T},
\end{align*}
which implies that
\begin{equation*}
Ind_{i,m}(W^{m-1}_T)=\mathbb C\textendash\textrm{span}\{
x_d\otimes_{\mathbb CP_{m-1}}x_{d_T}\ |\ d\in P_m,\ \beta(d)=\widetilde{T}^i\},
\end{equation*}
where $\widetilde{T}^i=(T_0,\cdots,T_i\cup\{m\},\cdots,T_n)\in X_m$. The following theorem is then an immediate consequence.
\begin{theorem}~\label{induction}
For $T\in X_{m-1}$ and $1\leqslant i\leqslant n$, $Ind_{i,m}(W^{m-1}_T)\cong W^{m}_{\widetilde{T}^i}$ as $\mathbb CP_m$-modules.
\end{theorem}
For $M\in ob(\mathscr{C}_m)$, we define $Res_{0,m}(M)=(e_{m-1}\otimes I_0)M$ and restrict the action of $\mathscr C_m$ to $\mathscr C_{m-1}$
on $Res_{0,m}(M)$, through the embedding of algebras
\begin{equation*}\psi_{0,m}:\ \mathbb CP_{m-1}\longrightarrow\mathbb CP_m
\end{equation*}
which takes $x$ to $x\otimes I_0$. This gives a restriction functor
\begin{equation*}
Res_{0,m}:\ \mathscr C_{m}\longrightarrow\mathscr C_{m-1}.
\end{equation*}
Via the embedding $\psi_{0,m}$, a induction functor $Ind_{0,m}:\ \mathscr C_{m-1}\longrightarrow\mathscr C_{m}$ is also given by
$Ind_{0,m}(N)=\mathbb CP_m\otimes_{\mathbb CP_{m-1}}N$. Similarly one can show the following theorem.
\begin{theorem}~\label{resind0}
\begin{itemize}
\item[(1)]For $T\in X_m$, $Res_{0,m}(W^m_T)=0$ if $T_0=\emptyset$; otherwise
$Res_{0,m}(W^m_T)\cong W^m_{m_0-1,m_1,\cdots,m_n}$ as $\mathbb CP_{m-1}$-modules.
\item[(2)] For $S\in X_{m-1}$, $Ind_{0,m}(W^{m-1}_S)\cong W^m_{m_0+1,m_1,\cdots,m_n}$ as $\mathbb CP_{m}$-modules, where $m_i=|S_i|$.
\end{itemize}
\end{theorem}

It is easy to show that for $0\leqslant i\leqslant n$, $Ind_{i,m}$ is left
adjoint to $Res_{i,m}$, i.e.
\begin{equation*}
\textrm{Hom}_{\mathbb CP_{m}}(Ind_{i,m}(M),N)\cong \textrm{Hom}_{\mathbb CP_{m-1}}(M,Res_{i,m}(N)),
\end{equation*}
which is natural in $M\in ob(\mathscr{C}_{m-1})$ and $N\in ob(\mathscr{C}_{m})$.

\section{Crystal Structure}

\subsection{Crystal Bases for $U_q(gl_n)$-modules}

Fix an indeterminant $q$. The quantized enveloping algebra associated to $gl_{n+1}(\mathbb C)$, denoted by $U_q(gl_{n+1})$, is a $\mathbb C(q)$-algebra generated by $E_i$, $F_i$ and $q^h$ for $1\leqslant i\leqslant n$
 and $h\in\mathfrak h$. We omit the generating relations of $U_q(gl_{n+1})$ and refer the readers to \cite{HK}
for more details.
 
The representation theory
of $U_q(gl_{n+1})$ is paralleled to that of $gl_{n+1}(\mathbb C)$. It is known that finite dimensional irreducible polynomial representation
of $U_q(gl_{n+1})$, or $U_q(gl_{n+1})$-modules, are indexed by the set of partitions of the form
\begin{equation*}
\lambda=(\lambda_1, \lambda_2, \cdots, \lambda_{n+1}),
\end{equation*}
where $\lambda_i\in\mathbb N$ and $\lambda_1\geqslant\lambda_2\geqslant\cdots\geqslant\lambda_{n+1}\geqslant 0$.
Indeed, every finite dimensional irreducible $U_q(gl_{n+1})$-module $V$ has a decomposition of weight spaces
\begin{equation*}
V=\oplus_{\mu\in\mathfrak h^{\ast}}V_{\mu},
\end{equation*}
where $V_{\mu}=\{v\in V\ |\ q^hv=\lambda(h)v\}$. Furthermore, there is a highest weight $\lambda$ among all $\mu$'s such that $V_{\mu}\neq 0$ in
the following sense,
\begin{equation*}
\lambda-\mu\in\sum_{i=1}^{n}\mathbb N(\epsilon_i-\epsilon_{i+1}).
\end{equation*}
The polynomial representations we considered are those
whose weights can be written as a combination of $\epsilon_i$ with non-negative integer coefficients. The highest weight $\lambda$ of an irreducible polynomial
representation is known to be dominant, i.e.
\begin{equation*}
\lambda=\lambda_1\epsilon_1+\cdots+\lambda_{n+1}\epsilon_{n+1}\in\mathfrak h^{\ast},
\end{equation*}
where $\lambda_i\in\mathbb N$ and $\lambda_1\geqslant\lambda_2\geqslant\cdots\geqslant\lambda_{n+1}\geqslant 0$. Hence each dominant
weight corresponds to a partition, or, a Young diagram and we do not distinguish them.  We use $V(\lambda)$ to denote the
finite dimensional irreducible polynomial $U_q(gl_{n+1})$-module of highest weight weight $\lambda$. It is well known that the basis vectors of
$V(\lambda)$ can be parameterized by semistandard Young tableaux of shape $\lambda$\cite{Fu, HK}.
The crystal basis theory developed by Kashiwara 
not only reinterprets this classical result, but also give a combinatorial explanation to the Littlewood-Richardson rule. Indeed, one can realize the
crystal graph of $B(\lambda)$ with semistandard Young tableaux of shape $\lambda$. For example, when $\lambda=\epsilon_1$, $V(\epsilon_1)$ is the natural
$U_q(gl_{n+1})$-module of dimension $n+1$ with the following crystal graphs
\begin{equation*}
\Yvcentermath1\young(0)\xrightarrow{\ \ \tilde{f}_1\ \ }\young(1) \xrightarrow{\ \ \tilde{f}_2\ \ }\young(2) \xrightarrow{\ \
 \tilde{f}_3\ \ }
\cdots\xrightarrow{\ \tilde{f}_{n}\ }\young(n),
\end{equation*}
where $wt(\Yvcentermath1\young(j))=\epsilon_{j+1}$, $\varepsilon_i(\Yvcentermath1\young(j) )=\delta_{i,j}$ and $\varphi_i(\Yvcentermath1\young(j) )=\delta_{i,j+1}$.
 Mention that for convenience, the Young tableaux in this paper are
Young diagrams filled with numbers in $\{0, 1, \cdots, n\}$ instead of those in $\{1,\cdots, n+1\}$ though the latter is more often used. 
Since every irreducible polynomial $U_q(gl_{n+1})$-module $V(\lambda)$ is a
direct summand of $V(\epsilon_1)^{\otimes |\lambda|}$, there is an embedding of crystals from $B(\lambda)$ to $B(\epsilon_1)^{\otimes|\lambda|}$. Note
that usually the ways of embedding are not unique, among which we choose the following one, called the Middle-Eastern reading,
\begin{equation*}
\Psi: B(\lambda)\longrightarrow B(\epsilon_1)^{\otimes|\lambda|},
\end{equation*}
which maps a semistandard Young tableau $T\in B(\lambda)$ to the tensor product of boxes in $T$ from right to left in each row and from top to bottom.
For instance,
\begin{equation*}
\Psi(\Yvcentermath1\young(0113,23,3))=\Yvcentermath1\young(3)\otimes\Yvcentermath1\young(1)\otimes\Yvcentermath1\young(1)\otimes\Yvcentermath1\young(0)
\otimes\Yvcentermath1\young(3)\otimes\Yvcentermath1\young(2)\otimes\Yvcentermath1\young(3).
\end{equation*}
Lifting the actions of Kashiwara operators $\tilde{e_i}$ and $\tilde{f}_i$ on $B(\epsilon_1)^{|\lambda|}$ to those on $B(\lambda)$, one obtains a crystal stucture
on $B(\lambda)$.

\subsection{$B(m\epsilon_1)$}

In this subsection, we categorify the crystal basis $B(\lambda)$ for $\lambda=m\epsilon_1$ via representation theory
of colored planar rook algebras. In this case, $\lambda$ can be viewed as a Young diagram which has only one row of $m$ boxes. As discussed in
last subsection, we can realize $B(m\epsilon_1)$ with semistandard Young tableaux of shape $(m,0,\cdots,0)$.

We remind the readers that in Section 3, for $0\leqslant i\leqslant n$, two functors
\begin{align*}
Res_{i,m}:\ &\mathscr{C}_m\longrightarrow \mathscr{C}_{m-1},\\
Ind_{i,m}:\ &\mathscr{C}_{m-1}\longrightarrow \mathscr{C}_{m},
\end{align*}
are introduced.
We denote by $[\mathscr{C}_m]$ the set of all isomorphism classes of irreducible objects in $\mathscr{C}_m$,
i.e.
\begin{equation*}
[\mathscr{C}_m]=\{[W^m_{m_0,m_1,\cdots,m_n}]\ |\ m_i\geqslant 0,\ \sum_{i=0}^{n}m_i=m \}.
\end{equation*}
Then there are two maps induced by the functors $Res_{i,m}$ and $Ind_{i,m}$,
\begin{align*}
[Res_{i,m}]:\ &[\mathscr{C}_m]\longrightarrow [\mathscr{C}_{m-1}]\cup \{0\},\\
[Ind_{i,m}]:\ &[\mathscr{C}_{m-1}]\longrightarrow [\mathscr{C}_{m}].
\end{align*}
For $1\leqslant i\leqslant n$, we define two functors $\tilde{\mathscr{E}}_{i,m},\ \tilde{\mathscr{F}}_{i,m}: \mathscr{C}_m\longrightarrow\mathscr{C}_m$
\begin{align*}
\tilde{\mathscr{E}}_{i,m}&=Res_{i,m+1}\comp Ind_{i-1,m+1},\\
\tilde{\mathscr{F}}_{i,m}&=Res_{i-1,m+1}\comp Ind_{i,m+1}.
\end{align*}
Let $\tilde{e}_{i,m}$ and $\tilde{f}_{i,m}$ be the maps induced by $\tilde{\mathscr{E}}_{i,m}$ and $\tilde{\mathscr{E}}_{i,m}$ respectively, i.e.
\begin{align*}
\tilde{e}_{i,m}=[\tilde{\mathscr{E}}_{i,m}]: &[\mathscr{C}_m]\longrightarrow [\mathscr{C}_{m}]\cup \{0\},\\
\tilde{f}_{i,m}=[\tilde{\mathscr{F}}_{i,m}]: &[\mathscr{C}_m]\longrightarrow [\mathscr{C}_{m}]\cup \{0\}.
\end{align*}
By Theorem~\ref{reduction},~\ref{induction} and~\ref{resind0}, we have the following proposition.
\begin{proposition}~\label{operator}
For $1\leqslant i\leqslant n$ and $m_0, \cdots, m_n\in\mathbb N$ with $m=\sum_{i=0}^{n}m_i$,
\begin{align*}
\tilde{e}_{i,m}([W^m_{m_0,\cdots,m_n}])&=
\begin{cases}
[W^m_{m_0,\cdots,m_{i-1}+1,m_i-1,\cdots,m_n}] & \text{if}\ m_{i}>0\\
0 & \text{if}\ m_{i}=0.
\end{cases}\\
\tilde{f}_{i,m}([W^m_{m_0,\cdots,m_n}])&=
\begin{cases}
[W^m_{m_0,\cdots,m_{i-1}-1,m_{i}+1,\cdots,m_n}] & \text{if}\ m_{i-1}>0\\
0 & \text{if}\ m_{i-1}=0.
\end{cases}
\end{align*}
\end{proposition}
We define the weight of $[W^m_{m_0,\cdots,m_n}]$ to be
\begin{equation*}
m_0\epsilon_1+m_1\epsilon_2+\cdots+m_n\epsilon_{n+1},
\end{equation*}
denoted by $wt([W^m_{m_0,\cdots,m_n}])$.
For $1\leqslant i\leqslant n$, two maps $\varepsilon_{i,m}, \varphi_{i,m}:\ [\mathscr C_{m}]\longrightarrow\mathbb N$ are defined by
\begin{align*}
\varepsilon_{i,m}([W^m_{m_0,\cdots,m_n}])&=m_i,\\
\varphi_{i,m}([W^m_{m_0,\cdots,m_n}])&=m_{i-1}.
\end{align*}
\begin{proposition}
$[\mathscr C_{m}]$ with $(wt, \tilde{e}_{i,m}, \tilde{f}_{i,m}, \varepsilon_{i,m}, \varphi_{i,m})$ is a crystal.
\end{proposition}
The proposition can be checked directly by Proposition~\ref{operator}. It remains to show
 that $[\mathscr C_m]$ is isomorphic to $B(m\epsilon_1)$ as crystals.

We consider the map
\begin{align*}
\rhoup:\ [\mathscr{C}_m]&\longrightarrow B(m\epsilon_1)\\
          [W^m_{m_0,\cdots,m_n}]&\longmapsto  \Yvcentermath1\young(0\cdots01\cdots1\cdots n\cdots n),
\end{align*}
where the number of $i$'s in the tableau is $m_i$ for $0\leqslant i\leqslant n$. Observe that $\rhoup$ is a bijection and \
it preserves the weight. We only show that $\rhoup$ commutes with the Kashiwara operator $\tilde{e}_i$, i.e.
\begin{equation}~\label{commute}
\rhoup\tilde{e}_{i,m}=\tilde{e}_i\rhoup,
\end{equation}
and the case for $\tilde{f}_i$ is similar. The proof proceeds by induction on $m$. It is trivial when $m=1$ and we assume that (\ref{commute}) holds
if $m<k$. Consider the case when $m=k$. If $m_i=0$, we have
\begin{equation*}
\tilde{e}_{i,m}([W^m_{m_0,\cdots,m_n}])=0
\end{equation*}
and it only needs to show that
\begin{equation}~\label{0}
\tilde{e}_i(\Yvcentermath1\young(pq\cdots r))=0
\end{equation}
where $\Yvcentermath1\young(pq\cdots r)$ denotes
the image of $[W^m_{m_0,\cdots,m_n}]$ under $\rhoup$ with $0\leqslant p\leqslant q\leqslant\cdots\leqslant r\leqslant n$.
Through the Middle-Eastern reading, (\ref{0}) is equivalent to the equality
$\tilde{e}_i(\Yvcentermath1\young(r)\otimes\cdots\otimes\young(q)\otimes\young(p))=0$,
or
\begin{equation}~\label{6}
\tilde{e}_i(\Yvcentermath1\young(q\cdots r)\otimes\young(p))=0.
\end{equation}
Since there is no $i$ in the tableau $\Yvcentermath1\young(pq\cdots r)$,
we have $p\neq i$ and hence
\begin{equation*}
\varphi_i(\Yvcentermath1\young(q\cdots r))\geqslant\varepsilon_i(\Yvcentermath1\young(p))=0.
\end{equation*}
It implies that $\tilde{e}_i(\Yvcentermath1\young(q\cdots r)\otimes\young(p))=\tilde{e}_i(\Yvcentermath1\young(q\cdots r))\otimes\young(p)$
and (\ref{6}) follows by the assumption.

Next, Suppose $m_i\geq 0$. In order to prove (\ref{commute}), we only need to show $\tilde{e}_i$ acts on $\Yvcentermath1\young(pq\cdots r)$ by
changing the leftmost $i$ to $i-1$. Similarly, we consider the image of $\Yvcentermath1\young(q\cdots r)\otimes\young(p)$ under $\tilde{e}_i$.
If $p=i$, one has $q\geqslant p=i$ and by assumption,
\begin{equation*}
\varphi_i(\Yvcentermath1\young(q\cdots r))=\varphi_{i,m}([W^{m-1}_{0,\cdots, 0, m_i-1, \cdots, m_n}])=0<\varepsilon_i(\Yvcentermath1\young(p))=1.
\end{equation*}
Hence $\tilde{e}_i(\Yvcentermath1\young(q\cdots r)\otimes\young(p))=\Yvcentermath1\young(q\cdots r)\otimes\tilde{e}_i(\young(i))$. If $p\neq i$, then
$p<i$. It implies
\begin{equation*}
\varphi_i(\Yvcentermath1\young(q\cdots r))\geqslant 0=\varepsilon_i(\Yvcentermath1\young(p)),
\end{equation*}
and then $\tilde{e}_i(\Yvcentermath1\young(q\cdots r)\otimes\young(p))=\tilde{e}_i(\Yvcentermath1\young(q\cdots r))\otimes\young(p)$ which completes
the proof by the assumption.
\begin{theorem}
$[\mathscr{C}_m]\cong B(m\epsilon_1)$ as crystals.
\end{theorem}

\subsection{Tensor Product of Crystals}
Let $\lambda=(\lambda_1,\cdots,\lambda_k)$ be a composition of an integer $l$, i.e.
\begin{equation*}
\sum_{i=1}^{k}\lambda_i=l,
\end{equation*}
where $\lambda_i$ is a positive integer for $1\leqslant i\leqslant k$. We denote by $\mathbb CP_{\lambda}$
the algebra
\begin{equation*}
\mathbb CP_{\lambda_1}\otimes\cdots\otimes\mathbb CP_{\lambda_k},
\end{equation*}
which can be viewed as a subalgebra of $\mathbb CP_{l}$. The following result is straightforward.
\begin{proposition}~\label{tensemi}
$\mathbb CP_{\lambda}$ is semisimple and  $\mathbb CP_{\lambda}$-modules of the form
\begin{equation*}
W^{\lambda_1}_{\lambda_{10},\cdots,\lambda_{1n}}\otimes\cdots\otimes W^{\lambda_k}_{\lambda_{k0},\cdots,\lambda_{kn}}
\end{equation*}
with $\lambda_i=\sum_{j=0}^{n}\lambda_{ij}$ exhaust all finite dimensional irreducible $\mathbb CP_{\lambda}$-modules up to isomorphism.
\end{proposition}
Let $\mathscr C_{\lambda}$ be the category of finite dimensional $\mathbb CP_{\lambda}$-modules and let $[\mathscr C_{\lambda}]$ be the set of
all isomorphism classes of irreducible modules in $\mathscr C_{\lambda}$. By Proposition~\ref{tensemi},
\begin{equation*}
[\mathscr C_{\lambda}]=\{[W^{\lambda_1}_{\lambda_{10},\cdots,\lambda_{1n}}\otimes\cdots\otimes W^{\lambda_k}_{\lambda_{k0},\cdots,\lambda_{kn}}]
\ |\ \lambda_i=\sum_{j=0}^{n}\lambda_{ij},\ \lambda_{ij}\in\mathbb N\}.
\end{equation*}

We define suitable operators $\tilde{e}_{i,\lambda}$ and $\tilde{f}_{i,\lambda}$ on $[\mathscr C_{\lambda}]$ so that it admits a crystal structure. For
$1\leqslant i\leqslant n$, we assign a sequence of $-$'s and $+$'s to each
$[M]=[W^{\lambda_1}_{\lambda_{10},\cdots,\lambda_{1n}}\otimes\cdots\otimes W^{\lambda_k}_{\lambda_{k0},\cdots,\lambda_{kn}}]$ as the following
\begin{equation}~\label{isig0}
(\overbrace{\underbrace{-,\cdots,-}_{\lambda_{1i}},\underbrace{+,\cdots,+}_{\lambda_{1,i-1}}}^{the\ 1^{st}\ component},\cdots,
\overbrace{\underbrace{-,\cdots,-}_{\lambda_{ki}},\underbrace{+,\cdots,+}_{\lambda_{k,i-1}}}^{the\ k^{th}\ component}).
\end{equation}
Deleting all $(+,-)$-pairs in the above sequence, we obtain a new sequence of the form
 \begin{equation}~\label{isig}
(-,\cdots,-,+,\cdots,+).
\end{equation}
Take $1\leqslant s,t\leqslant k$ such that the rightmost $-$ and the leftmost $+$ in (\ref{isig}) belong to the $s^{th}$ and the $t^{th}$ component in the original
sequence (\ref{isig0}) respectively. Then define
\begin{align}~\label{rule}
\tilde{e}_{i,\lambda}([M])
&=[W^{\lambda_1}_{\lambda_{10},\cdots,\lambda_{1n}}\otimes\cdots\otimes\tilde{\mathscr E}_{i,\lambda_s}(W^{\lambda_s}_{\lambda_{s0},\cdots,\lambda_{s
n}}) \otimes\cdots\otimes W^{\lambda_k}_{\lambda_{k0},\cdots,\lambda_{kn}}],\\
\tilde{f}_{i,\lambda}([M])
&=[W^{\lambda_1}_{\lambda_{10},\cdots,\lambda_{1n}}\otimes\cdots\otimes\tilde{\mathscr F}_{i,\lambda_t}(W^{\lambda_t}_{\lambda_{t0},\cdots,\lambda_{t
n}}) \otimes\cdots\otimes W^{\lambda_k}_{\lambda_{k0},\cdots,\lambda_{kn}}].
\end{align}
Also, we define
\begin{align}
wt([M])&=\sum_{j=1}^{k}wt([W^{\lambda_j}_{\lambda_{j0},\cdots,\lambda_{jn}}]),\\
\varepsilon_{i,\lambda}([M])&=\textrm{max}\{j\geqslant 0\ |\ \tilde{e}_{i,\lambda}^j([M])\neq 0\},\\
\varphi_{i,\lambda}([M])&=\textrm{max}\{j\geqslant 0\ |\ \tilde{f}_{i,\lambda}^j([M])\neq 0\}.
\end{align}
See that there is a natural bijection between $[\mathscr C_{\lambda}]$ and the Cartesian product
$[\mathscr C_{\lambda_1}]\times\cdots\times[\mathscr C_{\lambda_k}]$. Since $[\mathscr C_{\lambda_i}]\cong B(\lambda_i\epsilon_1)$ as crystals
and the actions of $\tilde{e}_{i,\lambda}$ and $\tilde{f}_{i,\lambda}$ on $[\mathscr C_{\lambda}]$ coincide with those of $\tilde{e}_{i}$ and
$\tilde{f}_i$ on the tensor product
$[\mathscr C_{\lambda_1}]\otimes\cdots\otimes[\mathscr C_{\lambda_k}]$\cite{???}, we have the following theorem.
\begin{theorem}
For a composition $\lambda=(\lambda_1,\cdots,\lambda_k)$ of $l$, there is an isomorphism of crystals
$[\mathscr C_{\lambda}]\cong B(\lambda\epsilon_1)\otimes\cdots\otimes B(\lambda_k\epsilon_1)$.
\end{theorem}
Furthermore, if $\lambda=(\lambda_1,\cdots,\lambda_k)$ is a partition of $l$ with length $k\leqslant n+1$, with the help of the Middle-Eastern reading, it is not hard to deduce that
the connected component of the
crystal graph $[\mathscr C_{\lambda}]$ which contains $[W^{\lambda_1}_{\lambda_1,0,\cdots,0}\otimes W^{\lambda_2}_{0,\lambda_2,0, \cdots,0}
\otimes\cdots\otimes W^{\lambda_k}_{0,\cdots,0,\lambda_k,0,\cdots,0}]$ is isomorphic to $B(\lambda)$ as crystals.

\end{document}